\documentclass[12pt]{amsart}

\usepackage{a4wide}

\usepackage{enumerate}
\usepackage{xcolor}

\newtheorem{theorem}{Theorem}[section]

\newtheorem{proposition}[theorem]{Proposition}

\newtheorem*{maintheorem1}{Theorem 1}
\newtheorem*{maintheorem2}{Theorem 2}

\theoremstyle{definition}

\newcommand{\cal}[1]{\ensuremath{\mathcal{#1}}}
\newcommand{\II}{\ensuremath{I\!I}}

\DeclareMathOperator{\Ric}{Ric}
\DeclareMathOperator{\trace}{Tr}

\begin{document}

\title[Some problems on ruled hypersurfaces in nonflat complex space forms]{Some problems on ruled hypersurfaces\\ in nonflat complex space forms}

\author[O.~P\'erez-Barral]{Olga P\'erez-Barral}
\address{Department of Mathematics, University of Santiago de Compostela, Spain.}
\email{olgaperez.barral@usc.es}

\thanks{The author acknowledges support by projects MTM2016-75897-P (AEI/FEDER, Spain) and ED431C 2019/10 (Xunta de Galicia, Spain), and by a research grant under the Ram\'on y Cajal project RYC-2017-22490 (AEI/FSE, Spain).} 



\begin{abstract}
We study ruled real hypersurfaces whose shape operators have constant squared norm in nonflat complex space forms. In particular, we prove the nonexistence of such hypersurfaces in the projective case.
We also show that biharmonic ruled real hypersurfaces in nonflat complex space forms are minimal, which provides their classification due to a known result of Lohnherr and Reckziegel. 
\end{abstract}


\maketitle

\section{Introduction}

A ruled real hypersurface in a nonflat complex space form, that is, in a complex projective or hyperbolic space, $\mathbb{C}P^{n}$ or $\mathbb{C}H^{n}$, is a submanifold of real codimension one which is foliated by totally geodesic complex hypersurfaces of $\mathbb{C}P^{n}$ or $\mathbb{C}H^{n}$. Ruled hypersurfaces in nonflat complex space forms constitute a very large class of real hypersurfaces. It becomes then an interesting problem to classify these objects under some additional geometric properties. For example, Lohnherr and Reckziegel classified ruled minimal hypersurfaces in nonflat complex space forms into three classes~\cite{LR}: Kimura type hypersurfaces in $\mathbb{C}P^{n}$ or $\mathbb{C}H^{n}$, bisectors in $\mathbb{C}H^{n}$ and Lohnherr hypersurfaces  in $\mathbb{C}H^{n}$. Moreover, they proved that Lohnherr hypersurfaces in $\mathbb{C}H^{n}$ are the only complete ruled hypersurfaces with constant principal curvatures in nonflat complex space forms~\cite{LR}. 

Another important notion in the context of real hypersurfaces is that of Hopf hypersurface, which is defined as a real hypersurface whose Reeb vector field is an eigenvector of the shape operator at every point. Ruled hypersurfaces in nonflat complex space forms are never Hopf; indeed, the smallest tangent distribution invariant under the shape operator and containing the Reeb vector field has rank two. In particular, the minimal ruled hypersurfaces mentioned above have an additional property, which has been introduced in \cite{DDV:annali}: they are strongly 2-Hopf, that is, the smallest distribution invariant under the shape operator and containing the Reeb vector field is integrable and has rank two, and the principal curvatures associated with the principal directions defining such distribution are constant along its integral submanifolds. This concept is also important since it characterizes, at least in the complex projective and hyperbolic planes $\mathbb{C}P^{2}$ and $\mathbb{C}H^{2}$, the real hypersurfaces of cohomogeneity one that can be constructed as union of principal orbits of a polar action of cohomogeneity two on the ambient space.

Motivated by some recent results concerning the classification of ruled hypersurfaces in nonflat complex space forms having constant mean curvature \cite{holi} or having constant scalar curvature \cite{Maeda_2}, we firstly focus on studying ruled hypersurfaces in $\mathbb{C}P^{n}$ or $\mathbb{C}H^{n}$ whose shape operators have constant squared norm, proving that there are no such hypersurfaces in $\mathbb{C}P^{n}$, whereas any possible example in $\mathbb{C}H^{n}$ must be strongly 2-Hopf.

	\begin{maintheorem1}
		Let $M$ be a ruled real hypersurface in a nonflat complex space form whose shape operator has constant norm. Then, $M$ is a strongly 2-Hopf real hypersurface in a complex hyperbolic space.	In particular, there are no ruled hypersurfaces in complex projective spaces whose shape operator has constant norm.
	\end{maintheorem1}

We note that, in the hyperbolic case, the Lohnherr hypersurfaces (as homogeneous ruled hypersurfaces) are examples of ruled real hypersurfaces with shape operator of constant norm. The problem of deciding whether these are the only such examples remains open.

\medskip

A hypersurface is said to be biharmonic if its defining isometric immersion is a biharmonic map, that is, a smooth map which is a critical point of the so-called bienergy functional (see, for example, \cite{ou} for more information on biharmonic hypersurfaces). It is well known that any minimal hypersurface is biharmonic and there exist some known results and conjectures claiming that the converse is, under certain circumstances, also true. For example, Chen conjectured that any minimal hypersurface in the Euclidean space is biharmonic. In the context of Riemannian manifolds of nonpositive curvature, it has been proved that both compact biharmonic hypersurfaces and biharmonic hypersurfaces with constant mean curvature are exactly the minimal ones \cite{jiang,ou}. However, if one removes these conditions 
one cannot ensure (in principle) minimality. Indeed, deciding whether, in general, biharmonicity implies minimality in ambient spaces of nonpositive curvature is the content of the generalized Chen's conjecture, proposed by Caddeo, Montaldo and Oniciuc in \cite{caddeo}. Ou and Tang have constructed some counterexamples which prove that this conjecture is not true \cite{outang}. However, it is still one of the main motivations for studying biharmonic hypersurfaces in the setting of Riemannian manifolds of nonpositive curvature due to the incompleteness of the examples provided by these two authors. Then, it becomes interesting to study biharmonic hypersurfaces satisfying other conditions, such as ruled ones, particularly in complex hyperbolic spaces, which are negatively curved. We point out that a general study of biharmonic submanifolds of arbitrary codimension in complex space forms has been developed in \cite{fetcu}.

It has been recently proved (see \cite{toru}) that biharmonic ruled hypersurfaces in complex projective spaces are minimal. Our second goal in this article is to extend this result to the entire context of nonflat complex space forms. In particular, we prove the following result.

\begin{maintheorem2}
	Let $M$ be a biharmonic ruled real hypersurface in a nonflat complex space form. Then, $M$ is minimal, and therefore an open part of one of the following hypersurfaces:
	\begin{enumerate}
		\item a Kimura type hypersurface in a complex projective or hyperbolic space, or
		\item a bisector in a complex hyperbolic space, or
		\item a Lohnherr hypersurface in a complex hyperbolic space.
	\end{enumerate}
\end{maintheorem2}

\section{Preliminaries}\label{sec:preliminaries}

In this section we introduce the notation and terminology that we are going to use throughout this article. For more information, we refer to \cite{CR} and \cite{NR}.

Let $\bar{M}$ be a Riemannian manifold and $M$ a smooth hypersurface of $\bar{M}$. Since the arguments that follow are local, we can assume that $M$ is embedded and take a unit normal vector field $\xi$ on $M$. We denote by $\langle\,\cdot\,,\,\cdot\,\rangle$ the metric of $\bar{M}$, by $\bar{\nabla}$ its Levi-Civita connection and by $\bar{R}$ its curvature tensor. Let $\nabla$ be the Levi-Civita connection of $M$. Then, the relation between $\nabla$ and $\bar{\nabla}$ is given by the Gauss formula $\bar{\nabla}_{X}Y=\nabla_{X}Y+\II(X,Y)$, where $\II$ denotes the second fundamental form of $M$. 

The shape operator $S$ of $M$ is the endomorphism of the tangent bundle of $M$ given by $SX=-(\bar{\nabla}_{X}\xi)^{\top}$, where $X$ is a tangent vector to $M$ and $(\cdot)^\top$ denotes the orthogonal projection onto the tangent space to $M$. As the shape operator is a self-adjoint endomorphism with respect to the induced metric on $M$, it can be diagonalized with real eigenvalues and orthogonal eigenspaces. Each eigenvalue $\lambda$ is called a principal curvature and its corresponding eigenspace, $T_{\lambda}$, is a principal curvature space. 

In this paper we will need some second order equations of submanifold geometry. In particular, we will use the Codazzi equation
\begin{equation*}
\langle \bar{R}(X,Y)Z,\xi\rangle = \langle(\nabla_{X}S)Y,Z\rangle-\langle(\nabla_{Y}S)X,Z\rangle,
\end{equation*}
as well as the Gauss equation
\begin{equation*}
\langle\bar{R}(X,Y)Z,W\rangle = \langle R(X,Y)Z,W\rangle + \langle SX,Z\rangle\langle SY,W\rangle -\langle SX,W\rangle\langle SY,Z\rangle,
\end{equation*}
where $X,Y,Z,W\in\Gamma(TM)$ are sections of $TM$.

In what follows, we will assume that the ambient manifold $\bar{M}$ is a complex space form of complex dimension $n$ and constant holomorphic sectional curvature $c\in\mathbb{R}$, $\bar{M}^{n}(c)$. It is known that its curvature tensor $\bar{R}$ is given by
	\begin{align*}
	\langle\bar{R}(X,Y)Z,W\rangle={}&
	\frac{c}{4}\Bigl(
	\langle Y,Z\rangle\langle X,W\rangle
	-\langle X,Z\rangle\langle Y,W\rangle\\[-1ex]
	&\phantom{\frac{c}{4}\Bigl(}
	+\langle JY,Z\rangle\langle JX,W\rangle
	-\langle JX,Z\rangle\langle JY,W\rangle
	-2\langle JX,Y\rangle\langle JZ,W\rangle
	\Bigr),
	\end{align*}
where $J$ denotes the complex structure of $\bar{M}^{n}(c)$. Moreover, since $\bar{M}^{n}(c)$ is a K\"ahler manifold, $\bar{\nabla}J=0$.

Now let $M$ be a real hypersurface of $\bar{M}^{n}(c)$, that is, a real submanifold of real codimension one. The tangent vector field $J\xi$ is usually called the Hopf or Reeb vector field of $M$. We define the integer-valued function $h$ on $M$ as the number of the principal curvature spaces onto which $J\xi$ has nontrivial projection.

A real hypersurface $M$ of $\bar{M}^{n}(c)$ is said to be ruled if it is foliated by totally geodesic complex hypersurfaces of $\bar{M}^{n}(c)$. Equivalently, $M$ is ruled if the orthogonal distribution to the Hopf vector field $J\xi$ is integrable and its leaves are totally geodesic submanifolds of $\bar{M}^{n}(c)$. 
Locally, ruled hypersurfaces are embedded, but globally, they may have self-intersections and singularities. See~\cite{Maeda} and~\cite[Section 8.5.1]{CR} for more information on ruled real hypersurfaces. 

We finally recall the notion of strongly 2-Hopf real hypersurface \cite{DDV:annali}. A real hypersurface $M$ in $\bar{M}^{n}(c)$ is said to be strongly 2-Hopf if the following conditions hold:
\begin{enumerate}
	\item The smallest $S$-invariant distribution $\mathcal{D}$ of $M$ that contains the Hopf vector field $J\xi$ has rank 2.
	\item $\mathcal{D}$ is integrable.
	\item The spectrum of $S|_{\mathcal{D}}$ is constant along the integral submanifolds of $\mathcal{D}$.
\end{enumerate}
Notice that the first condition is equivalent to $h=2$ and that a hypersurface satisfying both (1) and (2) is said to be a 2-Hopf hypersurface \cite[Section 8.5.1]{CR}.

\section{Proof of the Main Theorems}
 
Let $M$ be a ruled real hypersurface in a nonflat complex space form $\bar{M}^{n}(c)$, $c\neq0$, with (locally defined) unit normal vector field $\xi$. 

We briefly recall some facts from~\cite[Section 3]{holi}. It is known that there exists an open and dense subset $\cal{U}$ of $M$ where $h=2$. $\cal{U}$ has exactly two nonzero principal curvature functions $\alpha$ and $\beta$, both of multiplicity one at every point, and $J\xi=aU+bV$ for some unit vector fields $U\in\Gamma(T_\alpha)$ and $V\in \Gamma(T_\beta)$ and nonvanishing smooth functions $a$, $b\colon \cal{U}\to\mathbb{R}$ satisfying $a^2+b^2=1$ and
\begin{equation}\label{eq:1}
a^2=\frac{\alpha}{\alpha-\beta}, \qquad b^2=\frac{\beta}{\beta-\alpha}.
\end{equation}
 From now on, we will work in the open and dense subset $\cal{U}$ of $M$. 

With this notation, it can be proved~\cite[Lemma~3.1]{DD:indiana} that there exists a unit vector field $A\in\Gamma(T_0)$ such that
\begin{align}\label{eq:alg}
J\xi&=aU+bV, &JU&=-bA-a\xi, &JV&=aA-b\xi, &JA&=bU-aV.
\end{align}

Taking these expressions into account, we obtain the following.

\begin{proposition}\label{prop:Levi-Civita}
		Let $M$ be a ruled hypersurface in a nonflat complex space form. Then, its Levi-Civita connection satisfies the following equations:
	\begin{align*}
		&\langle\nabla_{U}U,V\rangle=\frac{V(\alpha)}{\alpha-\beta},
		&&\langle\nabla_{V}V,U\rangle=-\frac{U(\beta)}{\alpha-\beta},\\
		&\langle\nabla_{U}U,A\rangle=\frac{4A(\alpha)-3abc}{4\alpha},
		&&\langle\nabla_{V}V,A\rangle=\frac{4A(\beta)+3abc}{4\beta},\\
		&\langle\nabla_{U}V,A\rangle=\frac{3c}{4(\alpha-\beta)}+\alpha-\frac{aA(\alpha)}{b\alpha},
		&&\langle\nabla_{V}U,A\rangle=\frac{3c}{4(\alpha-\beta)}-\beta-\frac{bA(\beta)}{a\beta},\\
		&\langle\nabla_{A}U,V\rangle=\frac{ac\beta-4a\alpha\beta^{2}-4b\alpha A(\beta)}{4a\beta(\alpha-\beta)},
		&&\nabla_{A}A=0.
	\end{align*}		
Moreover, for any unit vector field orthogonal to $A$ in the 0-principal curvature distribution, $W\in\Gamma(T_{0}\ominus\mathbb{R}A)$, the following relations hold:
\begin{align*}
&\langle\nabla_{U}U,W\rangle=\frac{W(\alpha)}{\alpha},
&&\langle\nabla_{V}V,W\rangle=\frac{W(\beta)}{\beta},
&&\langle\nabla_{W}U,V\rangle=\frac{b\alpha W(\beta)}{a\beta(\beta-\alpha)}.
\end{align*}
In addition, 
\begin{align*}
	&U(\beta)=-\frac{\beta^{2}
	 U(\alpha)+2ab\alpha(\alpha-\beta)V(\beta)}{3\alpha\beta}, 
	&&V(\alpha)=\frac{2ab\beta(\alpha-\beta)U(\alpha)-\alpha^{2}V(\beta)}{3\alpha\beta},\\
	&A(\alpha)=\frac{b(\alpha-\beta)(a\beta(4\alpha\beta-c)+2b\alpha A(\beta))}{2\beta^{2}},
	&&W(\alpha)=-\frac{\alpha}{\beta}W(\beta).
\end{align*}
\end{proposition}

\begin{proof}
	Since $U$ and $V$ are orthogonal eigenvectors of the shape operator $S$ associated with the eigenvalues $\alpha$ and $\beta$, respectively, we have
	\begin{align*}
	\langle (\nabla_{U}S)V, U\rangle &= \langle \nabla_{U}(SV)-S\nabla_{U}V, U\rangle = \langle \nabla_{U}(\beta V), U\rangle -  \alpha\langle \nabla_{U}V, U\rangle\\
	&=\langle U(\beta)V+\beta\nabla_{U}V, U\rangle -\alpha \langle\nabla_{U}V,U \rangle= -(\alpha-\beta)\langle \nabla_{U}V, U\rangle.
	\end{align*}
	As $U$ has constant length, $\langle \nabla_{V}U,U\rangle=0$, and thus, proceeding as before,
	\begin{align*}
	\langle (\nabla_{V}S)U, U\rangle &= \langle \nabla_{V}(SU)-S\nabla_{V}U, U\rangle = \langle \nabla_{V}(\alpha U), U\rangle - \alpha\langle \nabla_{V}U, U\rangle\\
	&=\langle V(\alpha)U+\alpha\nabla_{V}U,U\rangle-\alpha\langle \nabla_{V}U, U\rangle=V(\alpha).
	\end{align*}
	Using the expression for the curvature tensor of a complex space form and the relations~\eqref{eq:alg}, we obtain that $\langle\bar{R}(U,V)U,\xi\rangle=0$. Then, using the previous relations to apply the Codazzi equation to the triple $(U,V,U)$, we get
	\begin{equation}\label{eq:cod_uvu}
	0=V(\alpha)+(\alpha-\beta)\langle \nabla_{U}V,U\rangle,
	\end{equation}
	which gives the first relation in the statement.
	Analogously, the Codazzi equation applied to the triple $(V,U,V)$ yields
	\begin{equation}\label{eq:cod_vuv}
	0=U(\beta)-(\alpha-\beta)\langle\nabla_{V}U,V\rangle,
	\end{equation}
	which is equivalent to the second relation in the statement.
	
	Since $\bar{\nabla}J=0$, using the definition of the shape operator and the relations $J\xi=aU+bV$ and $\langle\nabla_{U}U,U\rangle=0$, we obtain
	\begin{align*}
	U(a)&= U\langle J\xi,U\rangle = \langle \bar{\nabla}_{U}J\xi,U\rangle+\langle J\xi,\bar{\nabla}_{U} U\rangle = \langle J\bar{\nabla}_{U}\xi,U\rangle+\langle J\xi,\nabla_U U\rangle\\
	&= \langle SU,JU \rangle + \langle aU+bV,\nabla_{U}U\rangle = b\langle V,\nabla_{U}U \rangle = -b\langle \nabla_{U}V,U \rangle.
	\end{align*}
	By multiplying this expression by $2a$ and taking into account that 
	\begin{equation*}
	2aU(a)=U(a^{2})=U\left(\frac{\alpha}{\alpha-\beta}\right)=\frac{\alpha U(\beta)-\beta U(\alpha)}{(\alpha-\beta)^{2}},
	\end{equation*}
	we get, using \eqref{eq:cod_uvu}, 
	\begin{equation}\label{eq:deriv_ua}
	0=\beta U(\alpha)-\alpha U(\beta)+2ab(\alpha-\beta)V(\alpha).
	\end{equation}
	Analogously, expanding the relation $V(a)=V\langle J\xi,U\rangle$, we deduce, inserting \eqref{eq:cod_vuv}, that
	\begin{equation}\label{eq:deriv_va}
	0=\beta V(\alpha)-\alpha V(\beta)-2ab(\beta-\alpha)U(\beta).
	\end{equation}
	
	
	Equations \eqref{eq:deriv_ua} and \eqref{eq:deriv_va} constitute a linear system in the unknowns $U(\beta)$ and $V(\alpha)$. After some calculations using~\eqref{eq:1}, we get that the determinant of the matrix of this system vanishes if and only in $\alpha\beta$ does, which cannot occur in $\mathcal{U}$. Then, there exists a unique solution given by
	\begin{align}\label{eq:ubeta,valpha}
	&U(\beta)=-\frac{2ab\alpha(\alpha-\beta)V(\beta)+\beta^{2}U(\alpha)}{3\alpha\beta}, &&V(\alpha)=\frac{2ab\beta(\alpha-\beta)U(\alpha)-\alpha^{2}V(\beta)}{3\alpha\beta}.
	\end{align}

	Now, proceeding as above, the Codazzi equation applied to the triples $(U,A,U)$, $(V,A,V)$, $(A,U,A)$ and $(A,V,A)$ yields
	\begin{align}\label{eq:cod_A}
	&\langle\nabla_{U}U,A\rangle=\frac{4A(\alpha)-3abc}{4\alpha}, &&\langle \nabla_{V}V,A\rangle=\frac{3abc+4A(\beta)}{4\beta},\\\nonumber
	&\langle\nabla_{A}A,U\rangle=0,
	&&\langle\nabla_{A}A,V\rangle=0.
	\end{align}

	Since $\bar{\nabla}J=0$ and $J\xi=aU+bV$, expanding the relations $0=U\langle J\xi,A\rangle$ and $0=V\langle J\xi,A\rangle$, inserting the expressions for $\langle\nabla_{U}A,U\rangle$ and $\langle \nabla_{V}A,V\rangle$ that follow from~\eqref{eq:cod_A}, and using~\eqref{eq:1}, we have
	\begin{align}\label{eq:deriv_A}
	&\langle\nabla_{U}V,A\rangle= \frac{3c}{4(\alpha-\beta)}+\alpha-\frac{aA(\alpha)}{b\alpha},
	&&\langle \nabla_{V}A,U \rangle =\frac{bA(\beta)}{a\beta}+\beta-\frac{3c}{4(\alpha-\beta)}.
	\end{align}

	Now, the Codazzi equation applied to the triple $(V,A,U)$ yields, after inserting the expression for $\langle\nabla_{V}A,U\rangle$ given in \eqref{eq:deriv_A},
	\begin{align*}
	0&=\langle \bar{R}(V,A)U,\xi\rangle-\langle(\nabla_V S)A,U\rangle+\langle(\nabla_A S)V,U\rangle\\
	&=\frac{c(2\alpha+\beta)}{4(\alpha-\beta)}+\alpha\langle\nabla_{V}A,U\rangle-(\alpha-\beta)\langle\nabla_{A}V,U\rangle
	=-\frac{c}{4}+\alpha\beta+(\beta-\alpha)\langle\nabla_{A}V,U\rangle+\frac{b\alpha A(\beta)}{a\beta},
	\end{align*}
	from where
	\begin{equation}\label{eq:cod_vau}
		\langle\nabla_{A}V,U\rangle=\frac{4b\alpha A(\beta)+4a\alpha\beta^{2}-ac\beta}{4a\beta(\alpha-\beta)}.
	\end{equation}
	Similarly, applying the Codazzi equation to the triple $(U,V,A)$ and using the expressions for $\langle\nabla_{U}V,A\rangle$ and $\langle\nabla_{V}U,A\rangle$ given by~\eqref{eq:deriv_A}, we obtain
	\begin{align*}
	0&=\langle \bar{R}(U,V)A,\xi\rangle-\langle (\nabla_U S)V,A\rangle+\langle(\nabla_V S)U,A\rangle\\&=-\frac{c}{4}-\beta\langle\nabla_{U}A,V\rangle+\alpha\langle\nabla_{V}A,U\rangle=\frac{c}{2}-2\alpha\beta+\frac{a\beta A(\alpha)}{b\alpha}-\frac{b\alpha A(\beta)}{a\beta},
	\end{align*}
	from where, using~\eqref{eq:1}, we get the following relation between $A(\alpha)$ and $A(\beta)$:
	\begin{equation}\label{eq:a_alpha}
	A(\alpha)=\frac{b(\alpha-\beta)(a\beta(4\alpha\beta-c)+2b\alpha A(\beta))}{2\beta^{2}}.
	\end{equation}

	Finally, let $W\in\Gamma(T_{0}\ominus\mathbb{R}A)$ be an arbitrary unit vector field orthogonal to $A$ in the 0-principal curvature distribution. Using the expressions for $J\xi, JU$ and $JA$ given in~\eqref{eq:alg}, the Codazzi equation applied to the triple $(A,U,W)$ yields $\langle\nabla_{A}U,W\rangle=0$, from where we deduce, using \eqref{eq:cod_A}, that $\nabla_{A}U$ is proportional to $V$. In particular, since  $T_{0}\ominus\mathbb{R}A$ is a complex distribution, $\langle\nabla_{A}U,JW\rangle=0$. Expanding the relation $0=A\langle JU,W\rangle$, and taking the previous fact into account, as well as the expression for $JU$ given in~\eqref{eq:alg}, one gets
	\begin{equation*}
		b\langle\nabla_{A}A,W\rangle=\langle\nabla_{A}U,JW\rangle=0.
	\end{equation*}
	Then, $\langle\nabla_{A}A,W\rangle=0$ and, in fact, using~\eqref{eq:cod_A}, we obtain that $\nabla_{A}A=0$.
	
	Proceeding as above, the Codazzi equation applied to the triples $(U,W,U)$ and $(V,W,V)$ yields
	\begin{align}\label{eq:cod_w}
	&\langle \nabla_{U}U,W\rangle=\frac{W(\alpha)}{\alpha}, &&\langle \nabla_{V}V,W\rangle=\frac{W(\beta)}{\beta}.
	\end{align}

	Expanding the relations $0=U\langle J\xi,W\rangle$ and $0=V\langle J\xi,W\rangle$ and inserting the expressions for $\langle\nabla_{U}U,W\rangle$ and $\langle\nabla_{V}V,W\rangle$ given by~\eqref{eq:cod_w}, we have
	\begin{align}\label{eqs_w}
	&\langle\nabla_{U}V,W\rangle=-\frac{aW(\alpha)}{b\alpha}, &&\langle\nabla_{V}U,W\rangle=-\frac{bW(\beta)}{a\beta}.
	\end{align}

	After some calculations using~\eqref{eq:cod_w} and~\eqref{eqs_w}, the Codazzi equation applied to the triple $(V,W,U)$ yields
	\begin{equation}\label{eq:cod_vwu}
		\langle\nabla_{W}V,U\rangle=\frac{b\alpha W(\beta)}{a\beta(\alpha-\beta)}
	\end{equation}
	and, analogously, the Codazzi equation applied to the triple $(U,V,W)$ reads
	\begin{equation*}\label{eq:cod_uvw}
	0=\frac{a\beta W(\alpha)}{b\alpha}-\frac{b\alpha W(\beta)}{a\beta},
	\end{equation*}
from where we get, after using~\eqref{eq:1}, the last formula in the statement. 
\end{proof}

\subsection[Constant norm of $S$]{Ruled hypersurfaces whose shape operator has constant norm}
We firstly focus on the proof of Theorem 1. 

Let $k=\alpha^{2}+\beta^{2}$ denote the squared norm of the shape operator of $M$. Since by hypothesis $k$ is constant, $X(k)=0$ for each $X\in TM$. Thus, $\alpha X(\alpha)+\beta X(\beta)=0$, from where 
\begin{equation}\label{eq:X_beta}
X(\beta)=-\frac{\alpha X(\alpha)}{\beta}, \ \text{ for each } X\in TM.
\end{equation}
Taking this fact into account, one can rewrite some of the relations given in Proposition~\ref{prop:Levi-Civita} in an easier way.

\begin{proposition}\label{prop:Levi-Civita-k}
	Suppose that $\alpha\neq-\beta$ on an open subset of $\cal{U}$. Then, with the previous notations, the Levi-Civita connection of such open subset satisfies the following equations:
		\begin{align*}
			&\nabla_U U=-\frac{ab(8\alpha\beta^{2}+c(3\alpha+\beta))}{4\alpha(\alpha+\beta)}A,
			&\nabla_U V=\frac{c(3\alpha+\beta)+4\alpha(\alpha^{2}+\beta^{2})}{4(\alpha^{2}-\beta^{2})}A,
			\\
			&\nabla_V V=\frac{ab(8\alpha^{2}\beta+c(\alpha+3\beta))}{4\beta(\alpha+\beta)}A,
			&\nabla_V U=\frac{c(\alpha+3\beta)+4\beta(\alpha^{2}+\beta^{2})}{4(\alpha^{2}-\beta^{2})}A.
		\end{align*}
	Moreover:
	\begin{equation}\label{eq:deriv_functions}
	U(\alpha)=V(\alpha)=W(\alpha)=0, \qquad \text{and}\qquad A(\alpha)=\frac{ab\beta(c-4\alpha\beta)}{2(\alpha+\beta)},
	\end{equation}
	for any $W\in\Gamma(T_{0}\ominus\mathbb{R}A)$.
\end{proposition}

\begin{proof}
First of all, in order to prove that $U(\alpha)=V(\alpha)=0$, we rewrite the expressions for $U(\beta)$ and $V(\alpha)$ given in Proposition~\ref{prop:Levi-Civita} using the relation~\eqref{eq:X_beta}. Some calculations using~\eqref{eq:1} show that such equations are equivalent to:
\begin{align*}
	\beta(3\alpha^{2}-\beta^{2})\ U(\alpha)+2ab\alpha^{2}(\alpha-\beta)\ V(\alpha)&=0,\\
	-2ab\beta^{2}(\alpha-\beta)\ U(\alpha)+\alpha(3\beta^{2}-\alpha^{2})\ V(\alpha)&=0,
\end{align*}
which constitute a homogeneous linear system in the unknowns $U(\alpha)$ and $V(\alpha)$. The determinant of the matrix of such system can be easily deduced to be, using~\eqref{eq:1}, $-3\alpha\beta(\alpha^{2}-\beta^{2})^{2}$, which cannot vanish since $\alpha\beta\neq0$ and $\alpha\neq\pm\beta$ on an open subset of $\mathcal{U}$. Then, we conclude that $U(\alpha)=V(\alpha)=0$.

Again, using~\eqref{eq:X_beta}, we can rewrite the expression for $A(\alpha)$ given in Proposition~\ref{prop:Levi-Civita}. Some calculations using~\eqref{eq:1} lead us to conclude that $2(\alpha+\beta)A(\alpha)=ab\beta(c-4\alpha\beta)$, which is equivalent to the last formula in the statement.

Given $W\in\Gamma(T_{0}\ominus\mathbb{R}A)$, $W(\alpha)=-\alpha W(\beta)/\beta$ by Proposition~\ref{prop:Levi-Civita} and $W(\beta)=-\alpha W(\alpha)/\beta$ by~\eqref{eq:X_beta}. Then $W(\alpha)(1-\alpha^{2}/\beta^{2})=0$, from where we deduce, since $\alpha\neq\pm\beta$, that $W(\alpha)=0$.

Inserting the expressions for $U(\alpha),\ V(\alpha),\ W(\alpha)$ and $A(\alpha)$ that we have just obtained into the relations given in Proposition~\ref{prop:Levi-Civita} one gets, after some calculations involving~\eqref{eq:1}, the formulas for $\nabla_{U}U,\ \nabla_{U}V,\ \nabla_{V}U,\ \nabla_{V}V$ in the statement.
\end{proof}

We can now conclude the proof of Theorem 1. First of all notice that, if $\alpha=-\beta$ on $\mathcal{U}$, then $0=X(k)=X(2\alpha^{2})=4\alpha X(\alpha)$, which implies that $X(\alpha)=0$ for any $X\in T\mathcal{U}$. Since $X\in TM$ is arbitrary, we deduce that both $\alpha$ and $\beta$ have to be constant on $\mathcal{U}$, and by the density of $\mathcal{U}$, also on $M$.

Suppose now that there exists a point $p\in\mathcal{U}$ such that $\alpha(p)\neq-\beta(p)$. Then, in an open neighborhood of $p$, $\alpha\neq-\beta$. Taking the expressions given by \eqref{eq:deriv_functions} into account, the definition of the Lie bracket of $M$ yields
\begin{equation}\label{eq:[u,v]}
[U,V](\alpha)=U(V(\alpha))-V(U(\alpha))=0.
\end{equation}
On the other hand, using the fact that the Levi-Civita connection of $M$ is torsion-free, inserting the expressions for $\nabla_{U}V$ and $\nabla_{V}U$ given in Proposition~\ref{prop:Levi-Civita-k}, we obtain
\begin{equation}\label{eq:torsion_free}
	[U,V](\alpha)=(\nabla_{U}V)(\alpha)-(\nabla_{V}U)(\alpha)=\frac{c+2(\alpha^{2}+\beta^{2})}{2(\alpha+\beta)}A(\alpha).
\end{equation}
Then, either $A(\alpha)=0$ or $c+2(\alpha^{2}+\beta^{2})=c+2k=0$. If $A(\alpha)=0$ on an open subset, since $U(\alpha)=V(\alpha)=W(\alpha)=0$ for any $W\in\Gamma(T_{0}\ominus\mathbb{R}A)$, both $\alpha$ and $\beta$ must be constant, and thus, $M$ has constant principal curvatures on such open subset. Suppose now that $2k+c=0$ or, equivalently, that $k=-c/2$ on an open subset of $\mathcal{U}$. 

In the projective case, since $c>0$, the equation $\alpha^{2}+\beta^{2}=-c/2$ has no solution and, on the other hand, there is no ruled hypersurface in the complex projective case having constant principal curvatures \cite[Remark~5]{LR}. Therefore, there is no example of ruled hypersurface in $\mathbb{C}P^{n}$ whose shape operator has constant norm.

In the hyperbolic case, let $\mathcal{D}:=\mathrm{span}\{U,V\}$ be the smallest $S$-invariant distribution of $\cal{U}$ that contains $J\xi$, which clearly has rank 2. On the one hand, if an open subset of $M$ has constant principal curvatures, then it is an open part of a Lohnherr hypersurface \cite[Remark~5]{LR}, which is known to be strongly $2$-Hopf. This can be checked directly from Proposition~3.1: taking into account that it has constant principal curvatures $\alpha=\sqrt{-c}/2$, $\beta=-\sqrt{-c}/2$ and $0$, it follows that $[U,V]=\nabla_{U}V-\nabla_{V}U=0$, hence $\cal{D}$ is integrable, and moreover $\cal{D}(\alpha)=\cal{D}(\beta)=0$. On the other hand, if an open subset of $M$ satisfies $k=-c/2$, it follows from~\eqref{eq:[u,v]} that $\cal{D}$ is integrable and, by virtue of~\eqref{eq:deriv_functions} and the assumption that $\alpha^2+\beta^2$ is constant, again $\cal{D}(\alpha)=\cal{D}(\beta)=0$. Thus, in any case, $M$ is a strongly 2-Hopf hypersurface, which concludes the proof.

\subsection[Biharmonic hypersurfaces]{Ruled biharmonic hypersurfaces}
In this subsection we prove Theorem 2. Recall that a submanifold of a Riemannian manifold is said to be biharmonic if its defining isometric immersion is a critical point of the bienergy functional \cite{ou}. Moreover, they can be characterized as those submanifolds having vanishing bitension field. In the particular case of codimension one isometric immersions, that is, in the setting of biharmonic hypersurfaces, there exists an explicit formula which completely characterizes them.

\begin{proposition}\cite[Theorem 2.1]{ou}
	Let $\bar{M}$ be a Riemannian manifold and $M\subset\bar{M}$ a hypersurface with unit normal vector field $\xi$. $M$ is biharmonic if, and only if, it satisfies the following relations
	\begin{equation}\label{eq:biharmonic}
		\begin{cases}
			&\Delta H-H|S|^{2}+H\overline{\Ric}(\xi,\xi)=0,\\
			&2S(\nabla H)+H\nabla H-2H(\overline{\Ric}(\xi))^{\top}=0,
		\end{cases}
	\end{equation}
where $H=\trace(S)$ is the mean curvature of the hypersurface, $\nabla$ denotes the gradient, $\Delta$ is the Laplace-Beltrami operator of $M$, and $\overline{\Ric}$ denotes both the (0,2) and the (1,1) Ricci tensors of~$\bar{M}$.
\end{proposition}

In our context, $\bar{M}=\bar{M}^{n}(c)$ is a complex space form of constant holomorphic sectional curvature $c\neq 0$. In this case, the Ricci tensor of $\bar{M}$ satisfies $\overline{\Ric}(\xi,\xi)=c(n+1)/2$ and $(\overline{\Ric}(\xi))^{\top}=0$, which follows immediately from the formula for the curvature tensor of a complex space form.
Thus, in our case, equations~\eqref{eq:biharmonic} can be rewritten as follows (cf. \cite[Proposition~2.1]{fetcu}) :
\begin{equation}\label{eq:biharmonic2}
\begin{cases}
& \Delta H-H|S|^{2}+\frac{1}{2}Hc(n+1)=0,\\
& 2S(\nabla H)+H\nabla H=0.
\end{cases}
\end{equation}


We assume from now on that $M$ is a biharmonic ruled real hypersurface in a nonflat complex space form $\bar{M}^{n}(c)$. We will use the notations introduced in Section 3 for ruled real hypersurfaces. Thus, the mean curvature function of $M$ is $H=\alpha+\beta$.
 
As in the previous section, according to the discussion before Proposition~\ref{prop:Levi-Civita}, there is an open and dense subset $\cal{U}$ of $M$ where $h=2$. Proposition~\ref{prop:Levi-Civita} and relations~\eqref{eq:alg} hold in this open subset. In what follows, we will work in terms of an orthonormal basis of eigenvectors $\{U,V,A,W_{4},\dots,W_{2n-1}\}$, where $W_{i}\in\Gamma(T_0\ominus\mathbb{R}A)$, $i\in\{4,\dots,2n-1\}$.

Suppose that the mean curvature, $H=\alpha+\beta$, is not zero on an open subset of $\cal{U}$. We will work on this open subset of $M$ from now on. Since $M$ is a biharmonic hypersurface, it satisfies equations~\eqref{eq:biharmonic2}.

With respect to the orthonormal eigenbasis $\{U,V,A,W_{4},\dots,W_{2n-1}\}$, we have
\begin{equation*}
\nabla H=U(H)U+V(H)V+A(H)A+\sum\limits_{i=4}^{2n-1}W_{i}(H) W_{i}.
\end{equation*}
On the other hand, since $U, V, A$ and $W_{i}$, for $i\in\{4,\dots,2n-1\}$, are orthogonal eigenvectors of the shape operator $S$ of $M$ associated with eigenvalues $\alpha, \beta$ and 0, respectively, we have 
\begin{equation*}
S(\nabla H)=\alpha U(H)U+\beta V(H)V.
\end{equation*}
Thus, inserting these relations into the second equation in~\eqref{eq:biharmonic2}, we obtain
\begin{equation*}
0=2S(\nabla H)+H\nabla H=(2\alpha+H)U(H)U+(2\beta+H)V(H)V+HA(H)A+\sum\limits_{i=4}^{2n-1}HW_{i}(H) W_{i}.
\end{equation*}
Since $H\neq0$ by assumption, one can deduce that $A(H)=0$ and $W_{i}(H)=0$ for $i\in\{4,\dots,2n-1\}$. Moreover, one of the following conditions holds on an open subset:
\begin{enumerate}
	\item $U(H)=V(H)=0$.
	\item $\alpha=\beta=-H/2$.
	\item $U(H)=0$ and $2\beta+H=0$. 
	\item $V(H)=0$ and $2\alpha+H=0$. 
\end{enumerate}

Neither case (1) nor case (2) are possible. Indeed, if $U(H)=V(H)=0$ on an open subset, then such subset has constant mean curvature and, since it is ruled, $H=0$ \cite{holi}, which gives a contradiction. Analogously, since $M$ is ruled, $\alpha\neq\beta$ on any open subset.

Suppose that $U(H)=0$ and $2\beta+H=0$ or, equivalently, $H=2\alpha/3$ (case (4) is analogous). Then, both $\alpha$ and $\beta$ can be expressed as $\alpha=3H/2$ and $\beta=-H/2$, respectively. Inserting these expressions in the formula for $A(\alpha)$ given in Proposition~\ref{prop:Levi-Civita}, one gets
\begin{equation*}
	\frac{3}{2}A(H)=A(\alpha)=2b(a(c+3H^{2})-3bA(H)).
\end{equation*}
Moreover, $A(H)\neq0$ on $\mathcal{U}$, from where $ab(3H^{2}+c)=0$. Since $a$ and $b$ are not zero in $\mathcal{U}$, $M$ has constant mean curvature on $\mathcal{U}$ and, as $M$ is a ruled, it must be minimal (see \cite{holi}), which concludes the proof.

\end{document}